\documentclass[11pt]{article}
\usepackage{amsmath}
\usepackage{amsthm,amsfonts,amssymb,latexsym}
\usepackage[utf8]{inputenc}
\usepackage[english]{babel}
\usepackage{color}

\usepackage{hyperref} 
\hypersetup{
colorlinks=true, 
breaklinks=true, 
urlcolor= blue, 
linkcolor= blue, 
}
\newtheorem{theorem}{Theorem}[section]
\newtheorem{proposition}[theorem]{Proposition}

\newtheorem{fact}[theorem]{Fact}

\theoremstyle{definition}

\theoremstyle{remark}

\def\R{{\mathbb R}}
\def\N{{\mathbb N}}
\newcommand{\xR}{{]}{-\infty},+\infty]}
\newcommand{\Rex}{\xR}
\newcommand{\Rb}{\overline{\R}}

\newcommand{\Fcal}{\mathcal{F}}

\newcommand{\Lsc}{{\rm Lsc}}
\newcommand{\LC}{{\rm LC}}

\newcommand{\LACG}{{\rm LACG}}
\newcommand{\CLACG}{{\rm CLACG}}
\newcommand{\LL}{{\rm LL}}

\newcommand{\AUSS}{\LACG_*^{\natural{a}}}
\newcommand{\AFUSS}{\LACG^{\natural{a}}}
\newcommand{\NFUSS}{\LC^{\natural{n}}}

\newcommand{\gu}{g^{\natural}}
\newcommand{\fu}{f^{\natural}}
\newcommand{\gsu}{g^{\natural\natural}}
\newcommand{\fsu}{f^{\natural\natural}}

\newcommand{\ps}{\smallbreak}

\newcommand{\del}{\partial}

\newcommand{\delc}{\widehat{\del}}

\newcommand{\dom} {{\rm dom} \kern.15em}
\newcommand{\tq}{:}
\newcommand{\la}{\langle}
\newcommand{\ra}{\rangle}

\newcommand{\eps}{\varepsilon}

\newcommand{\bx}{\bar{x}}
\newcommand{\xb}{\bar{x}}
\newcommand{\yb}{\bar{y}}

\newcommand{\tow}{{\stackrel{w^*}{\longrightarrow}}\;}

\begin{document}
\thispagestyle{empty}
\begin{center}
{\large\bf\sc
Links between functions and subdifferentials}
\medskip\\
\today
\end{center}

\begin{center}
  {\small\begin{tabular}{c}
  Marc Lassonde\\
   Universit\'e des Antilles, Pointe \`a Pitre, and LIMOS, 
   Clermont-Ferrand, France\\
  E-mail: marc.lassonde@univ-antilles.fr
  \end{tabular}}
\end{center}

\medbreak\noindent
\textbf{Abstract.}
A function  in a class $\mathcal{F}(X)$
is said to be subdifferentially determined in $\mathcal{F}(X)$
if it is equal up to an additive constant
to any function in $\mathcal{F}(X)$ with the same subdifferential.
A function is said to be subdifferentially representable
if it can be recovered from a subdifferential.
We identify large classes of lower semicontinuous functions
that possess these properties.
\medbreak\noindent
  
\medbreak\noindent
\section{Introduction}\label{intro}
In this paper we are interested in two fundamental links between functions and subdifferentials:
the subdifferential determination and the
subdifferential representation of functions.
A function  in a class $\Fcal(X)$
is said to be \textit{subdifferentially determined} in $\Fcal(X)$
if it is equal up to an additive constant
to any function in $\Fcal(X)$ with the same subdifferential.
A function is said to be \textit{subdifferentially representable}
if it can be expressed in terms of a subdifferential, or put another way,
if it can be recovered from a subdifferential.

The subdifferential determination of
extended-real-valued lower semicontinuous convex functions
defined on Hilbert spaces was brought to light by Moreau \cite{Mor65}.
His result was later extended to general Banach spaces by
Rockafellar \cite{Roc70b}.
In the non-convex case, the first works are due to Rockafellar \cite{Roc82}
for the class of (Clarke) regular locally Lipschitz functions,
and to Poliquin \cite{Pol91} for the class of primal lower nice functions
with possibly extended-real values.
Since then, this property has been considered for various classes
of functions;
let us mention the works of Correa-Jofré \cite{CJ89},
Qi \cite{Qi89,Qi90}, Birge-Qi \cite{BQ93},
Thibault-Zagrodny \cite{TZ95,TZ05,TZ10},
Borwein-Moors \cite{BM98b},
Thibault-Zlateva \cite{TZla05}, Bernard-Thibault-Zagrodny \cite{BTZ05},
Zaj{\'{\i}}{\v{c}}ek \cite{Zaj12} and our recent work \cite{Las18b}.
The subdifferential representation of extended-real-valued
lower semicontinuous convex functions defined on a Banach space
was established by Rockafellar \cite{Roc70b};
different proofs of this result are proposed by Taylor \cite{Tay73},
Thibault \cite{Thi00} and Ivanov-Zlateva \cite{IV08},
and a refined version by Benoist-Daniilidis \cite{BD05}.
Few results exist for non-convex functions;
let us mention Qi \cite{Qi89} and Birge-Qi \cite{BQ93}.
\smallbreak
In the present article we study the subdifferential determination
and the subdifferential representation properties
with respect to an abstract subdifferential. 
Our abstract subdifferential recovers the Clarke, the Michel-Penot
and the Ioffe subdifferentials in any Banach space, and 
the elementary subdifferentials (proximal, Fr\'echet, Hadamard, \ldots),
as well as their viscosity and limiting versions,
in appropriate Banach spaces.
The subdifferential determination property is considered
for different classes of functions lying between the class $\LL(X)$
of locally Lipschitz functions and the class $\Lsc(X)$ of
extended-real-valued lower semicontinuous functions.
The subclasses of subdifferentially determined functions are
defined according to the continuity properties of their lower Dini derivative, independently of the
subdifferential.
In this introduction we give a brief overview of the
above-mentioned works and their connections with the present work.

The class $\LL(X)$ of locally Lipschitz functions
is the natural framework for many applications in
nonsmooth analysis and optimization.
However, further assumptions are in general necessary
to get sharpened results.
The most widely used  assumptions are \textit{regularity}
and \textit{semismoothness}:
the regularity introduced by Clarke \cite{Cla83} makes it possible
to obtain equality in the calculus rules for the (Clarke)
subdifferential, the semismoothness proposed by Mifflin
\cite{Mif77a,Mif77b} allows implementable algorithms for nonsmooth
constrained optimization. The common feature of these assumptions is
the postulation of a relationship between the lower Dini derivative
(hereinafter called the \textit{radial subderivative})
of $f$ at $\bx$,
\begin{align*}
f^r(\xb;u)&:=\liminf_{t\searrow 0}\,\frac{f(\xb+tu)-f(\xb)}{t},
\end{align*}
and the \textit{Clarke subdifferential} of $f$ at $\bx$,
\begin{gather*}
\partial_{C} f(\bx) := \{x^* \in X^* \tq \langle x^*,u\rangle \leq
f^\circ(\bx;u), \, \forall u \in X\},
\end{gather*}
where
\begin{align*}
f^\circ(\bx;u)&:=  
\limsup_{\substack{t \searrow 0\\x \to\bx}}\frac{f(x+tu) -f(x)}{t}.
\end{align*}
More precisely, a locally Lipschitz function $f:X\to \R$ is called
\textit{regular at a point $\xb$} provided for every $u\in X$,
\begin{equation}\label{intro1} 
f^r(\xb; u)=\max\{\la x^*,\xb\ra\tq x^*\in\partial_Cf(\xb)\}
(=f^\circ(\xb;u))
\end{equation}
and \textit{semismooth at $\xb$} provided for every $u\in X$,
\begin{equation}\label{intro2} 
x_n \to_u \xb \text{ and } x^*_n \in \partial_Cf(x_n) \Rightarrow
\la x^*_n,u\ra \to f^r(\xb; u),
\end{equation}
where $x_n \to_u \xb$ indicates that $x_n=\xb + t_n u_n$ with
$t_n\searrow 0$ and $u_n\to u$.
These two properties are independent: there are regular functions that are not semismooth, and semismooth functions that are not regular. The functions that are both regular and semismooth are characterized
by the submonotonicity of their (Clarke) subdifferential;
lower-$C_1$ functions are examples of such functions
(see Spingarn \cite{Spi81}).
It is well known that the properties of regularity and semismoothnes
can be reformulated in terms of the continuity of the radial subderivative;
indeed, (\ref{intro1}) and (\ref{intro2}) are respectively equivalent to
\begin{equation}\label{intro1b} 
f^r(\xb; u)=\limsup_{x\to \xb} f^r(x; u),\tag{\ref{intro1}b}
\end{equation}
\begin{equation}\label{intro2b} 
f^r(\xb; u)=\lim_{x\to_u \xb} f^r(x; u)\tag{\ref{intro2}b}
\end{equation}
(see for example \cite{Roc82} and \cite{CJ89}).
There is one significant difference between the formulas
(\ref{intro1b})-(\ref{intro2b}) and
the definitions (\ref{intro1})-(\ref{intro2}), namely,
the former can be extended beyond the Clarke subdifferential.
In this paper we will consider a property that generalizes both:
a locally Lipschitz function $f$ is said to be
\textit{upper semismooth at a point $\xb$, in the direction $u$}
\cite{Las18b}, if
\begin{equation}\label{intro3} 
f^r(\xb; u)=\limsup_{x\to_u \xb} f^r(x; u).
\end{equation}

Two stronger notions are naturally associated with
the Clarke subdifferential of locally Lipschitz functions.
We say that a locally Lipschitz function $f$ is \textit{strictly differentiable at $\xb$, in the direction $u$}
\cite{BM98a}, if
\begin{align*}
\lim_{\substack{t \searrow 0\\x \to\bx}}\frac{f(x+tu) -f(x)}{t}
\text{ exists};
\end{align*}
equivalently (see the proof of \cite[Proposition 2.2.4]{Cla83}),
\begin{align*}
f^\circ(\xb;u)=f^r(\xb;u)=-f^\circ(\xb;-u).
\end{align*}
We say that $f$ is \textit{strictly differentiable at $\xb$}, if $f$ is strictly differentiable at $\xb$ in every direction $u$;
equivalently, $\partial_C f(\xb)$ is a singleton
(\cite[Proposition 2.2.4]{Cla83}).
\smallbreak
Let us now go back to the subdifferential determination property in the
framework of locally Lipschitz functions. This property has been demonstrated for the Clarke subdifferential by Rockafellar \cite{Roc82}
on $\R^n$ for everywhere regular functions (Corollary 3 of Theorem 2);
by Correa-Jofré \cite{CJ89} on a Banach space, for densely Gateaux differentiable and everywhere regular functions (Proposition 4.1) and
for densely Gateaux differentiable and everywhere semismooth functions
(Proposition 5.4) (note that the hypothesis of dense Gateaux
differentiability is automatically satisfied on $\R^n$ by Rademacher's
theorem, on separable Banach spaces by Christensen's theorem,
on smooth spaces by Preiss' theorem); by Qi \cite{Qi89} on $\R^n$
for the so-called \textit{primal} functions, i.e.\ the functions that
are strictly differentiable almost everywhere and $D$-representable
(in fact, in $\R^n$ the functions that are strictly differentiable almost everywhere are always $D$-representable, see \cite[Corollary 3.10]{Bor91}
or \cite[Corollary 4.2]{BM97}); examples of primal functions
are the functions almost everywhere regular
\cite[Corollary 2 of Theorem 2]{Roc82} and the functions
everywhere semismooth \cite[Theorem 1]{Qi90}. All these results have been
generalized by Borwein-Moors (\cite[Theorem 3.8]{BM98b}), still for the
Clarke subdifferential, on arbitrary Banach spaces, for the so-called
\textit{essentially smooth} functions, i.e.\ the functions $f$ such that
for each $u$, $f$ is strictly differentiable in the direction $u$ almost
everywhere (in the sense of Haar).
The proof of \cite[Theorem 3.8]{BM98b} lies in the fact that the Clarke
subdifferential mapping of such functions is a minimal weak$^*$-cusco.
In Subsection \ref{variant}, we study the subdifferential determination
property in the class $\CLACG(G)$ of all real-valued Continuous functions
defined on a nonempty open convex subset $G$ of a Banach space $X$,
whose restrictions to Line segments $[a,b]\subset G$ are ACG
(see Section \ref{Dinisect} for the definition). This class includes the locally Lipschitz functions defined on $G$. In this class, we consider
the subclass $\CLACG^{\natural{ad}}(G)$
of \textit{densely almost upper semismooth} functions,
i.e.\ the functions $f$ satisfying the following property:
for every $\xb, u \in X$ with $[\xb,\xb+u]\subset G$
there exist sequences $\xb_n\to \xb$, $u_n\to u $ such that
for almost all $x_n \in [\xb_n,\xb_n+u_n[$,
$f^r(x_n; u_n)$ is finite and $f$ is upper semismooth at
$x_n$ in the direction $u_n$.
We show that this subclass contains the Borwein-Moors essentially
smooth functions (Proposition \ref{BMclass}) and is subdifferentially
determined for the abstract subdifferential in the class $\CLACG(U)$
(Theorem \ref{subdiffdetcont}).
This theorem appears as a continuous variant of
our main theorem which concerns lower semicontinuous functions.
\smallbreak
Let us now discuss the subdifferential determination property
for the class of extended-real-valued lower semicontinuous functions.
In this general context, the conditions \eqref{intro1b} and \eqref{intro3}
are no longer suitable. In fact, we must take into account the value
$f^r(x;\xb-x)$ which now does not necessarily tend towards $0$ when $x$
tends towards $\xb$: we must integrate this factor in the formulas. Similar
phenomena occur every time we are dealing with unbounded values; see,
for example, the definition \eqref{wclosure} below of the closure $\delc f$
of the subdifferential $\del f$ of a lower semicontinuous function $f$,
the discussion in \cite[Subsection 2.1]{JL02} on the closure of the
convex subdifferential, and the references therein.
Thus, the conditions \eqref{intro1b} and \eqref{intro3}
are respectively replaced by 
\begin{equation}\label{intro1bb}
f^r(\xb;u)=\inf_{\alpha\ge 0}\limsup_{x\to \xb} f^r(x; u +\alpha (\xb-x)),
\end{equation}
\begin{equation}\label{intro3b}
f^r(\xb;u)=
\inf_{\alpha\ge 0}\limsup_{x\to_u \xb} f^r(x; u +\alpha (\xb-x)).
\end{equation}
Clearly, the conditions \eqref{intro1bb} and \eqref{intro3b}
reduce to \eqref{intro1b} and \eqref{intro3}, respectively,
when the function $f$ is Lipschitz around $\xb$.

In \cite{Las18b} we have shown that the following
lower semicontinuous functions
$f$ satisfy \eqref{intro1bb} at every point of every segment
$[\xb,\xb+u[\subset\dom f$:
the convex functions,
the directionally approximately convex functions
and the $\del$-subdifferentially and directionally stable functions
of Thibault-Zagrodny \cite{TZ05}.
More precisely, all these functions belong to the subclass
$\Lsc^{\natural\natural}(X)$ of the functions $f\in\Lsc(X)$
with convex domain $\dom f$ such that
for every $[\xb,\xb+u]\subset \dom f$ and for all points
$\xb_t\in [\xb,\xb+u{[}$,
$f^r(\xb_t;u)<+\infty$ and \eqref{intro1bb} is satisfied.
Functions in $\Lsc^{\natural\natural}(X)$ have a powerful continuity
property along line segments, namely,
the restriction of every $f\in \Lsc^{\natural\natural}(X)$
to any line segment $[\xb,\xb+u]\subset \dom f$ is continuous at
the endpoints and locally Lipschitz at every point in ${]}\xb,\xb+u{[}$
(Proposition \ref{altFUSS}).
Finally, we show that the class $\Lsc^{\natural\natural}(X)$ is
subdifferentially determined in $\Lsc(X)$
(Theorem \ref{determination-revisited}\,(1)), thus extending
the special case $\gamma= 0$ in \cite[Theorem 4.1]{TZ05}.
A slight variant of Theorem \ref{determination-revisited}\,(1),
with a more complicated formulation, had been previously established
in \cite[Theorem 11]{Las18b}.

Next we consider the classes $\LC(X)$ and $\LACG(X)$ of
extended-real-valued lower semicontinuous functions whose
restrictions to line segments are, respectively, continuous and ACG.
In these (slightly) smaller classes, we expect to identify subclasses of
subdifferentially determined functions larger than the subclass
$\Lsc^{\natural\natural}(X)$. This is indeed the case! We sketch the
plan for the class $\LACG(X)$. We denote by $\LACG^{\natural{a}}(X)$
the subclass of those functions $f\in \LACG(X)$
with convex domain $\dom f$ such that
for every $[\xb,\xb+u]\subset \dom f$ and for almost all
points $\xb_t\in [\xb,\xb+u]$,
$f^r(\xb_t;u)$ is finite and \eqref{intro3b} is satisfied.
We show that $\LACG^{\natural{a}}(X)$ is indeed larger than
$\Lsc^{\natural\natural}(X)$, contains 
the lower semicontinuous
$\del$-essentially directionally smooth functions of Thibault-Zagrodny
\cite{TZ10} (Proposition \ref{eds-uss}) and
is subdifferentially determined in $\LACG(X)$
(Theorem \ref{determination-revisited}\,(3)). This extends
\cite[Theorem 3.4]{TZ10}.

Finally, we identify a large class of lower semicontinuous functions
which can be recovered from their (abstract) subdifferential via an
integration formula, namely, the functions $f\in \LACG_*(X)$
with convex domain $\dom f$ such that
for every $[\xb,\xb+u]\subset \dom f$ and for almost all
points $\xb_t\in [\xb,\xb+u]$, \eqref{intro3b} is satisfied
(Theorem \ref{recover}).
This subclass of subdifferentialy representable functions contains,
among other things, the 
extended-real-valued lower semicontinuous convex functions
considered by Rockafelar \cite{Roc70b} and the primal functions on $\R^n$
considered by Qi \cite[Theorem 9]{Qi89}. A
detailed description of this subclass is given in Subsection 5.3.

\smallbreak
It should be noticed that some interesting results mentioned above
are not recovered by our approach: the
subdifferential determination of primal lower nice functions
and the like,
studied in Poliquin \cite{Pol91}, Thibault-Zagrodny \cite{TZ95} and
Bernard-Thibault-Zagrodny \cite{BTZ05};
the local subdifferential determination of regular directionally Lipschitz
functions established by Thibault-Zlateva \cite{TZla05};
the subdifferential determination and the subdifferential representation
of locally Lipschitz functions, in finite dimensional spaces,
for the Michel-Penot subdifferential, given by Birge-Qi \cite{BQ93};
the subdifferential determination of the locally Lipschitz functions
that are essentially smooth on a generic line
parallel to a generic direction, in Asplund spaces, for the Clarke subdifferential, proven by Zaj{\'{\i}}{\v{c}}ek
\cite[Proposition 7.5]{Zaj12}.
\smallbreak
This paper is a continuation of our works \cite{Las18a,Las18b}.
In \cite{Las18a}, we have established a formula linking the radial subderivative
to other subderivatives and subdifferentials.
In \cite{Las18b}, we have proved a simple version of the subdifferential
determination property without resorting to measure and integration
theories. Here we propose a more precise statement of the subdifferential
determination problem and we provide new contributions based this time on
measure and integration theories. 
Moreover, we establish the subdifferential representation property
for a large class of extended-real-valued lower semicontinuous functions.
As in our paper \cite{Las18b}, the technique for demonstrating the
main theorems in the present paper
relies on our formula linking subderivative and subdifferential \cite{Las18a}
that reduces the original subdifferential problem on a Banach space to
a problem involving the radial subderivative on the real line.
The theory of ACG functions and Henstock-Kurzweil integrals is then used to
address this reduced problem. The relevant definitions and facts from this
theory are gathered in Section \ref{Dinisect}.
Subderivatives and subdifferentials are described in Section \ref{subsub}.
The main results are stated and proved in Section \ref{recovering}.
The last section contains examples and variants.

\section{Links between functions and subderivatives on the real line}\label{Dinisect}
In this section, we have compiled the relevant facts concerning the
ACG functions and the Henstock-Kurzweil integral
in connection with our subject.
Most of these facts have been taken from Gordon's textbook \cite{Gor94},
which offers a thorough analysis of this integral.
\smallbreak
Throughout the section, $[a,b]$ denotes a compact interval of $\R$,
$\varphi:\R\to\xR$ an extended-real-valued function and
$\dom \varphi:=\{x\in \R\tq \varphi(x)<+\infty\}$ its \textit{effective domain}.
The \textit{lower and upper right-hand Dini derivatives} of $\varphi$ 
at $t_0\in\dom \varphi$ are respectively given by
$$
D_+\varphi(t_0):=\liminf_{t\searrow 0}\frac{\varphi(t_0+t)-\varphi(t_0)}{t}, ~~~
D^+\varphi(t_0):=\limsup_{t\searrow 0}\frac{\varphi(t_0+t)-\varphi(t_0)}{t}.
$$
For points  $t_0\not\in\dom \varphi$, the Dini derivatives
are defined to be $-\infty$.
We say that $\varphi$ is \textit{right-differentiable} at $t_0$ if
$$
\lim_{t\searrow 0}\frac{\varphi(t_0+t)-\varphi(t_0)}{t}
$$
exists in $\R$, or equivalently, if $D_+\varphi(t_0)$ is finite and
$D_+\varphi(t_0)=D^+\varphi(t_0)$.
\smallbreak
Let $S\subset [a,b]$.
The function $\varphi$
is \textit{absolutely continuous (AC) on $S$}
if for each $\eps>0$ there exists $\delta >0$ such that
$\sum_{i=1}^n|\varphi(d_i)-\varphi(c_i)|<\eps$ whenever
$\{[c_i,d_i] : 1\le i\le n\}$ is a finite collection of
non-overlapping intervals that have endpoints in $S$
and satisfy $\sum_{i=1}^n(d_i-c_i)<\delta$.
The function $\varphi$ is \textit{generalized  absolutely continuous
(ACG) on $S$} if $\varphi|_S$ is continuous on $S$ and
$S$ can be written as a countable
union of sets on each of which $\varphi$ is AC.
\smallbreak
A property is said to hold \textit{almost everywhere in $S$},
or \textit{for almost all $t\in S$}, if it holds
in $S$ except for a set of (Lebesgue) measure zero.
As in \cite{Gor94}, we say that a property
holds \textit{nearly everywhere in $S$},
or \textit{for nearly all $t\in S$}, if it holds
 in $S$ except for a countable set.

\begin{theorem}[Subderivative test for monotonicity]
\label{monotone-test}
Let $\varphi:\R\to\xR$, $a\in \dom \varphi$ and $b\in\R$ with $b>a$.
Then $\varphi$ is nonincreasing on $[a,b]$ in each of the following situations:
\smallbreak {\rm (1)}
$\varphi$ is lower semicontinuous on $[a,b]$ and $D_+\varphi(t)\le 0$
everywhere in $[a,b{[}$,
\smallbreak {\rm (2)}
$\varphi$ is continuous on $[a,b]$ and $D^+\varphi(t)\le 0$
nearly everywhere in $[a,b]$,
\smallbreak {\rm (3)}
$\varphi$ is ACG on $[a,b]$ and $D_+\varphi(t)\le 0$
almost everywhere in $[a,b]$.
\end{theorem}

\begin{proof}
(1) First, we observe that $[a,b]\subset \dom \varphi$.
Indeed, let $c \in [a,b]$ and let $\mu\in\R$ such that
$\mu\le \varphi(c)-\varphi(a)$.
Applying the mean value inequality \cite[Lemma 4.1]{JL13}, we get a point
$x_0\in [a,c{[}$ such that
$\mu\le (c-a)D_+\varphi(x_0)\le 0$. Consequently,
$\varphi(c)-\varphi(a)\le 0$, hence $c\in \dom \varphi$.
Now, let $c,d\in  [a,b]$, $c< d$. Then,
$c,d\in \dom \varphi$.
Applying again the mean value inequality, we get a point
$x_0\in [c,d{[}$ such that
$\varphi(d)-\varphi(c)\le (d-c)D_+\varphi(x_0)\le 0$.
Hence, $\varphi(d)\le \varphi(c)$.

\smallbreak (2) This is a special case of \cite[Proposition 3]{Las18b}.

\smallbreak (3) See, e.g., \cite[Theorem 6.25]{Gor94}.\end{proof}

Let $\omega(\varphi,[c,d]):=
\sup\{|\varphi(y)-\varphi(x)|\tq c\le x\le y\le d\}$
denote the \textit{oscillation} of the function $\varphi$
on the interval $[c,d]$.
The function $\varphi:[a,b]\to\R$ is
\textit{absolutely continuous in the restricted sense (AC$_*$) on
$S\subset [a,b]$}
if for each $\eps>0$ there exists $\delta >0$ such that
$\sum_{i=1}^n\omega(\varphi,[c_i,d_i])<\eps$ whenever
$\{[c_i,d_i] : 1\le i\le n\}$ is a finite collection of
non-overlapping intervals that have endpoints in $S$ and
$\sum_{i=1}^n(d_i-c_i)<\delta$.
The function $\varphi$ is \textit{generalized
absolutely continuous in the restricted sense (ACG$_*$) on
$S\subset [a,b]$} if $\varphi|_S$ is continuous on $S$ and
$S$ can be written as a countable union of sets on each of which
$\varphi$ is AC$_*$.

If $\varphi$ is AC on the interval $[a,b]$, then
$\varphi$ is AC$_*$ on $[a,b]$; if $\varphi$ is continuous on $[a,b]$
and AC on every interval $[c,d]\subset {]}a,b{[}$, then
$\varphi$ is ACG$_*$ on $[a,b]$ since $\varphi$ is AC$_*$ on
each of the sets $\{a\}$, $\{b\}$ and $[a+1/n,b-1/n]$, $n\in\N$.
Clearly, ACG$_*$ implies ACG. The converse is not true:
an ACG function is not necessarily differentiable almost everywhere
\cite[Example 6.20]{Gor94} while an ACG$_*$ function is:

\begin{fact}[{\cite[Corollary 6.19, Exercise 7.9, Theorem 6.22]{Gor94}}]
\label{diff-ACG}
Let $\varphi:[a,b]\to\R$ be continuous.
\smallbreak
{\rm(1)} If $\varphi:[a,b]\to\R$ is ACG$_*$ on $[a,b]$, then $\varphi$ is
ACG on $[a,b]$ and
differentiable almost everywhere on $[a,b]$.
The converse is not true.
\smallbreak
{\rm(2)} If $\varphi$ is right-differentiable nearly everywhere on
$[a,b]$, then $\varphi$ is ACG$_*$ on $[a,b]$.
\end{fact}

A function $\varphi:[a,b]\to\R$ is
\textit{Henstock-Kurzweil integrable (HK-integrable) on $[a,b]$}
with integral $\Phi_{a,b}\in \R$ if
for each $\eps>0$ there exists a positive function
$\delta:[a,b]\to {]}0,+\infty[$ such that
$$\left|\sum_{i=1}^n \varphi(t_i)(d_i-c_i)-\Phi_{a,b}\right|<\eps$$
whenever
$\Pi:=\{(t_i,[c_i,d_i]) : 1\le i\le n\}$ is a finite collection of
non-overlapping tagged intervals (i.e. $t_i\in [c_i,d_i]$)
that satisfy $(d_i-c_i)<\delta(t_i)$ for all $i$.
If it exists, the integral $\Phi_{a,b}$ is uniquely defined
and we write
$$\Phi_{a,b}=\int_a^b \varphi(t)dt
:=\lim_{HK,\Pi}\sum_{i=1}^n \varphi(t_i)(d_i-c_i).$$
Moreover, if $\varphi$ is HK-integrable on $[a,b]$, then
$\varphi$ is HK-integrable on every subinterval of $[a,b]$
(see, e.g., \cite[Theorem 4]{Hen68} or \cite[Theorem 9.8]{Gor94}),
hence for every $x\in[a,b]$, the so-called indefinite HK-integral
$$\Phi_{a,x}=\int_a^x \varphi(t)dt$$
exists.
It turns out that a function $\Phi$ is an indefinite HK-integral on
$[a,b]$ if and only if $\Phi$ is ACG$_*$ on $[a,b]$.
More precisely:
 
\begin{fact}[{\cite[Theorem 9.17 and Exercise 11.7]{Gor94}}]
\label{ACG*-HK}
Let $\varphi,\Phi:[a,b]\to\R$ be two functions.
The following statements are equivalent: 
\smallbreak
{\rm(a)} $\varphi$ is HK-integrable on $[a,b]$ and
$\Phi(x)-\Phi(a)=\int_a^x \varphi(t)dt$ for all $x\in[a,b]$;
\smallbreak
{\rm(b)} $\Phi$ is ACG$_*$ on $[a,b]$ and 
$\Phi'(t)=\varphi(t)$ almost everywhere in $[a,b]$.
\end{fact}

In view of Fact \ref{ACG*-HK}, we can consider integrating functions
$\varphi$ that are only defined and finite almost everywhere in
$[a, b]$: such a function $\varphi:[a,b]\to\Rb$ is declared
HK-integrable on $[a,b]$ if there exists a finite-valued HK-integrable
function $\psi:[a,b]\to\R$
such that $\varphi(t)=\psi(t)$ almost everywhere in $[a,b]$.
Then $\int_a^x \varphi(t)dt:=\int_a^x \psi(t)dt$ and
the integral thus defined is independent of the chosen function $\psi$.
\begin{theorem}[Subderivative representation of functions]
\label{recov-subdiv-ACG}
A function $\varphi:[a,b]\to\R$ can be represented through
its subderivative via the HK-integration formula
\begin{equation}\label{subrep}
\varphi(x)-\varphi(a)=\int_a^x D_+\varphi(t)dt,\quad \forall x\in [a,b],
\end{equation}
if and only if $\varphi$ is ACG$_*$ on $[a,b]$.
\end{theorem}
\begin{proof}
By Fact \ref{ACG*-HK}, (a) $\Rightarrow$ (b), if $\varphi$ can be
represented as an indefinite HK-integral  on $[a,b]$, then
$\varphi$ is ACG$_*$ on $[a,b]$.
Conversely, if $\varphi$ is ACG$_*$ on $[a,b]$, then
by Fact \ref{diff-ACG}\,(1), its derivative $\varphi'(t)$ exists
almost everywhere in $[a,b]$, hence $D_+\varphi(t)=\varphi'(t)$
almost everywhere in $[a,b]$.
It therefore follows from Fact \ref{ACG*-HK}, (b) $\Rightarrow$ (a),
that the 
function $t\mapsto D_+\varphi(t)$
is HK-integrable on $[a,b]$ and $\varphi(x)-\varphi(a)=\int_a^x D_+\varphi(t)dt.$
\end{proof}
\section{Links between subderivatives and subdifferentials}\label{subsub}

From now on, $X$ is a real Banach space,
$B_X$ its unit ball, $X^*$ the topological dual,
and $\la .,. \ra$ the duality pairing.
For $x, y \in X$, we let $[x,y]:=\{ x+t(y-x) \tq t\in[0,1]\}$;
the sets $]x,y[$ and $[x,y[$ are defined accordingly.
Set-valued operators $T:X\rightrightarrows X^*$
are identified with their graph $T\subset X\times X^*$
and we write $\dom T:=\{x\in X\tq T(x)\ne\emptyset\}$.
All extended-real-valued functions $f : X\to\xR$ are assumed to be
lower semicontinuous (lsc)
and \textit{proper}, which means that
the set $\dom f:=\{x\in X\tq f(x)<\infty\}$ is non-empty.
A net $(x_\nu)_\nu\subset X$ is said to
\textit{converge to $\xb$ in the direction $v\in X$},
written $x_\nu\to_v \xb$,
if there are two nets $t_\nu\searrow 0$ (that is,
$t_\nu\to 0$ with $t_\nu>0$)
and $v_\nu\to v$ such that $x_\nu=\xb + t_\nu v_\nu$ for all $\nu$.
Observe that for $v=0$, $x_\nu\to_v\xb$ simply means $x_\nu\to\xb$.

\medbreak
The framework, terminology and notation are the same as in our works
\cite{Las18a,Las18b}. Let be given a lsc function $f:X\to\xR$,
a point $\xb\in\dom f$ and a direction $u\in X$.
We recall that the (lower right Dini) \textit{radial subderivative},
its upper version
and its upper strict version (the \textit{Clarke subderivative}) are
respectively defined by
\begin{align*}
f^r(\xb;u)&:=\liminf_{t\searrow 0}\,\frac{f(\xb+tu)-f(\xb)}{t},~~
f^r_+(\xb;u):=\limsup_{t\searrow 0}\,\frac{f(\xb+tu)-f(\xb)}{t},\\
f^\circ(\bx;u)&:=  
\limsup_{\substack{t \searrow 0\\(x,f(x)) \to (\bx,f(\bx))}}
\frac{f(x+tu) -f(x)}{t},
\end{align*}
whereas the (lower right Dini-Hadamard) \textit{directional subderivative}
and its upper strict version (the \textit{Clarke-Rockafellar subderivative})
are respectively given by: 
\begin{align*}
f^d(\bx;u)&:=
\liminf_{\substack{t \searrow 0\\u' \to u}}\frac{f(\bx+tu')-f(\bx)}{t},\\
f^\uparrow(\bx;u)&:= \sup_{\delta>0} 
\limsup_{\substack{t \searrow 0\\(x,f(x)) \to (\bx,f(\bx))}}
\inf_{u' \in B_{\delta}(u)}\frac{f(x+tu') -f(x)}{t}.
\end{align*}
For points  $\xb\not\in\dom f$, all the subderivatives
are defined to be $-\infty$.
\smallbreak
Besides these classical subderivatives, we 
shall also consider \textit{upper semicontinuous regularizations}
of the radial subderivative, in the \textit{directional sense} and
in the \textit{full sense}:
\begin{subequations}\label{natural}
\begin{align}
\fu(\xb;u):=\inf_{\alpha\ge 0}\limsup_{x\to_u \xb} f^r(x; u +\alpha (\xb-x)),
\label{naturala}\\
\fsu(\xb;u):=
\inf_{\alpha\ge 0}\limsup_{x\to \xb} f^r(x; u +\alpha (\xb-x)).
\label{naturalb}
\end{align}
\end{subequations}
It turns out that the regularized subderivatives $\fu$ and $\fsu$
can be expressed in terms of any bivariate function $f'$ lying between
the subderivatives $f^d$ and $f^\uparrow$
(point (3) below):
\begin{fact}[{\cite[Proposition 4, Proposition 7 and Theorem 3]{Las18a}}]
\label{nat-subdiv}
Let $f:X\to\xR$ be lsc on a Banach space $X$,
$\xb\in \dom f$ and $u\in X$.
\smallbreak
{\rm(1)}
If $f$ is convex, then $f^r(\xb;u)=\fsu(\xb;u)$.
\smallbreak
{\rm(2)} If $f(\xb)=\liminf_{t\searrow 0}f(\xb+tu)$,
then $f^r_+(\xb;u)\le \fu(\xb;u)$.
\smallbreak
{\rm(3)}
For any $f': X\times X\to \Rb$ such that
$f^d\le f'\le f^\uparrow$, 
\begin{subequations}\label{formula0}
\begin{align}
\fu(\xb;u)=\inf_{\alpha\ge 0}\limsup_{x\to_u\xb} f'(x;u+\alpha (\xb-x)),
\label{subderiv-formula0a}\\
\fsu(\xb;u)=\inf_{\alpha\ge 0}\limsup_{x\to\xb}\,f'(x;u+\alpha (\xb-x)).
\label{subderiv-formula0b}
\end{align}
\end{subequations}
\end{fact}
The relationships between the regularized and the classical
subderivatives are visualized on the following diagram
where $\rightarrow$ means `$\le$', $\xrightarrow{*}$ means
`$\le$ provided $f(\xb)=\liminf_{t\searrow 0}f(\xb+tu)$',
and $\xrightarrow{**}$ means `$\le$ provided $f$ is continuous
relative to its domain':
\begin{align}\label{diagram}
f^r(\xb;u)  & \rightarrow f^r_+(\xb;u)\xrightarrow{*}\fu(\xb;u)\rightarrow \fsu(\xb;u) \xrightarrow{**} f^\circ(\xb;u)\nonumber\\
\uparrow \quad &  \hspace{5cm}   \uparrow \quad\\
f^d(\xb;u)  & \qquad\qquad\quad\longrightarrow \qquad\qquad\quad f^\uparrow(\xb;u)\nonumber
\end{align}
If $f$ is locally Lipschitz at $\xb$ relative to its domain,
one has $f^r(\xb;u)=f^d(\xb;u)$,
$\fu(\xb;u)=\limsup_{x\to_u \xb} f^r(x; u)$ and
$\fsu(\xb;u)=\limsup_{x\to \xb} f^r(x; u)=f^\circ(\xb;u)=
f^\uparrow(\xb;u)$,
so the above diagram becomes a line:
\begin{align}\label{diagramlip}
f^r(\xb;u) =f^d(\xb;u) \rightarrow f^r_+(\xb;u)\rightarrow \fu(\xb;u)
\rightarrow \fsu(\xb;u)= f^\circ(\xb;u)=
f^\uparrow(\xb;u)\tag{\ref{diagram}Lip}.
\end{align}
If in addition $f$ is regular at $\xb$ in the sense of Clarke,
i.e.\ $f^r(\xb; u)=f^\circ(\xb;u)$,
then all the subderivatives coincide. But in general
the inequality $\fu(\xb;u)\leq\fsu(\xb;u)$ is strict:
for $f:x\in \R\mapsto -|x|$, $\xb=0$ and $u\ne 0$,
$$
\fu(0;u)=\limsup_{x\to_u 0} f^r(x;u)=-|u|<
\fsu(0;u)=\limsup_{x\to 0} f^r(x;u)=|u|.
$$
\medbreak
Next, given a lsc function $f:X\to\xR$ and a point $\xb\in\dom f$,
we recall that the \textit{Moreau-Rockafellar subdifferential}
(the subdifferential of convex analysis) and the \textit{Clarke subdifferential} are respectively defined by
\begin{gather*}
 \del_{MR} f (\xb) :=
 \{ x^* \in X^* \tq \la x^*,y-\xb\ra + f(\xb) \leq f(y),\, \forall y \in X \},\\
\partial_{C} f(\bx) := \{x^* \in X^* \tq \langle x^*,u\rangle \leq
f^\uparrow(\bx;u), \, \forall u \in X\}.
\end{gather*}
All the classical subdifferentials (proximal,
Fr\'echet, Hadamard, Ioffe, Michel-Penot, \ldots)
lie between these two objects,
and for a lsc convex $f$, all these subdifferentials coincide.
\ps
In the sequel, we call \textit{subdifferential} any operator $\del$ that associates 
a set-valued mapping $\partial f: X \rightrightarrows X^\ast$
to each function $f$ on $X$ so that
$$\del_{MR} f\subset \partial f\subset \partial_{C} f,$$
and the following \textit{Separation Principle} is satisfied:
\medbreak\noindent
(SP) \textit{For any lsc $f,\varphi$ with $\varphi$ convex Lipschitz
near $\xb\in\dom f $,
if $f+\varphi$ admits a local minimum at $\xb$, then
$0\in \delc f(\xb)+ \del \varphi(\xb),$ where
\begin{multline}\label{wclosure}
\delc f(\xb):= \{\, \xb^*\in X^*\tq \mbox{there is a net }((x_\nu,x^*_\nu))_\nu\subset \del f \mbox{ with }\\
       (x_\nu,f(x_\nu))\to (\bx,f(\bx)),\ x^*_\nu\tow \bx^*,\ \limsup_\nu\,\la x^*_\nu,x_\nu-\xb\ra\le 0\,\}.
\end{multline}}

The Clarke subdifferential, the Michel-Penot (moderate) subdifferential
and the Ioffe subdifferential satisfy (SP) in any Banach space.
The elementary subdifferentials (proximal, Fr\'echet, Hadamard, \ldots),
as well as their viscosity and limiting versions,
satisfy (SP) in appropriate Banach spaces.
See, e.g.\ \cite{Iof12,JL13,Pen13} and the references therein.
\ps
Subdifferentials satisfying the Separation Principle
(SP) are densely defined:
\begin{fact}[{\cite[Theorem 5.1]{JL13}}]\label{densitytheorem}
Let $X$ be a Banach space,
$f:X\to\xR$ be lsc and $\xb\in\dom f$.
Then, there exists a sequence $((x_n,x_n^*))_n\subset\del f$ such that
$x_n\to \xb$, $f(x_n)\to f(\xb)$ and $\limsup_n \la x_n^*, x_n-\bx\ra\le 0$.
\end{fact}

We call \textit{subderivative associated to a subdifferential $\del f$}
at a point $\xb\in \dom f$ in the direction $u\in X$ 
the \textit{support function} of the set $\del f(\xb)$ in the direction $u$,
which we denote by
$$
f^\del(\xb;u):=
\sup \,\{\la \xb^*,u \ra \tq \xb^*\in\del f(\xb)\}.
$$
A key feature is that the regularized subderivatives $\fu$ and $\fsu$
can also be expressed in terms of $f^\del$
for any subdifferential $\partial$:
\begin{fact}[{\cite[Theorem 3]{Las18a}}]
\label{nat-subdiff}
Let $f:X\to\xR$ be lsc on a Banach space $X$,
$\xb\in \dom f$ and $u\in X$.
For any subdifferential $\del$, 
\begin{subequations}
\begin{align}
\fu(\xb;u)=\inf_{\alpha\ge 0}\limsup_{x\to_u\xb} f^\del(x;u+\alpha (\xb-x)),
\label{Subdiff-formula0a}\\
\fsu(\xb;u)=\inf_{\alpha\ge 0}\limsup_{x\to\xb}\,f^\del(x;u+\alpha (\xb-x)).
\label{Subdiff-formula0b}
\end{align}
\end{subequations}
\end{fact}
\noindent As a straightforward consequence, we obtain a variant of
\cite[Proposition 7]{Las18b}:

\begin{theorem}[Subdifferential representation of the radial
subderivative]\label{recoversubdiv}
Let $f:X\to\xR$ be lsc on a Banach space $X$ and let $\del$ be a
subdifferential. Then, for any $\xb\in \dom f$ and $u\in X$,
\begin{subequations}\label{recoversubdivformula}
\begin{align}
f^r(\xb;u)=\fu(\xb;u)\Longleftrightarrow
f^r(\xb;u)=\inf_{\alpha\ge 0}\limsup_{x\to_u\xb} f^\del(x;u+\alpha (\xb-x)),
\label{recoversubdivformula0a}\\
f^r(\xb;u)=\fsu(\xb;u)\Longleftrightarrow
f^r(\xb;u)=\inf_{\alpha\ge 0}\limsup_{x\to\xb}\,f^\del(x;u+\alpha (\xb-x)).
\label{recoversubdivformula0b}
\end{align}
\end{subequations}
\end{theorem}

A lsc function $f:X\to\xR$ is declared \textit{upper semismooth}
(respectively, \textit{strictly upper semismooth})
at a point $\xb\in\dom f$ in the direction $u\in X$ provided
$\fu(\xb;u)= f^r(\xb;u)$ (respectively, $\fsu(\xb;u)= f^r(\xb;u)$)
--- the definitions used in \cite{Las18b} are slightly less demanding,
with $\le$ instead of $=$.
Examples of upper semismooth functions are
the locally Lipschitz Mifflin semismooth functions;
examples of strictly upper semismooth functions are
the locally Lipschitz Clarke regular functions,
the proper lsc (approximately) convex functions,
the lower-C$^1$ functions,
the Thibault-Zagrodny directionally stable functions.
See \cite{Las18b} and the last section of
the present paper for further discussion.
By Theorem \ref{recoversubdiv}, a function $f$ is
upper semismooth (respectively, strictly upper semismooth)
at a point $\xb\in \dom f$
in the direction $u$ if and only if its radial subderivative
$f^r(\xb;u)$ at $\xb$ in the direction $u$ can be recovered
from a subdifferential through the formula \eqref{recoversubdivformula0a}
(respectively, the formula \eqref{recoversubdivformula0b}).

\section{Links betweens functions and subdifferentials}
\label{recovering}

This section is devoted to the study of 
the subdifferential determination and the
subdifferential representation of a function.
We write $\Lsc(X)$ for the class of all lsc functions on $X$ and
$\LC(X)$ ($\LACG_*(X)$, $\LACG(X)$, respectively)
for the subclass of $\Lsc(X)$ consisting of lsc functions $f$
whose restrictions to Line segments $[a,b]\subset \dom f$ are Continuous
(ACG$_*$, ACG, respectively).

\smallbreak
For the subdifferential determination issue,
we consider three subclasses of $\Lsc(X)$
depending on the degree of regularity of the radial subderivative
of the functions, namely, the classes of strictly, nearly and
almost upper semismooth functions. They are respectively defined by: 
\begin{multline*}
\Lsc^{\natural\natural}(X)=
\{ g\in \Lsc(X)\tq \dom g \text{ is convex and }
\forall [\xb,\xb+u]\subset \dom g,\ \xb_t=\xb+tu,\\
g^r(\xb_t;u)<+\infty \text{ and }
\gsu(\xb_t;u)= g^r(\xb_t;u) \text{ for all } t\in [0,1{[}\}.
\end{multline*}
\vspace*{-0.8cm}
\begin{multline*}
\LC^{\natural{n}}(X)=
\{ g\in \LC(X)\tq \dom g \text{ is convex and }
\forall [\xb,\xb+u]\subset \dom g,\ \xb_t=\xb+tu,\\
g^r(\xb_t;u)\text{ is finite and } 
\gu(\xb_t;u)= g^r(\xb_t;u) \text{ for nearly all } t\in [0,1]\}.
\end{multline*}
\vspace*{-0.8cm}
\begin{multline*}
\LACG^{\natural{a}}(X)=
\{ g\in \LACG(X)\tq \dom g \text{ is convex and }
\forall [\xb,\xb+u]\subset \dom g,\ \xb_t=\xb+tu,\\
g^r(\xb_t;u)\text{ is finite and }
\gu(\xb_t;u)= g^r(\xb_t;u)\text{ for almost all } t\in [0,1]\}.
\end{multline*}

One has 
$\Lsc^{\natural\natural}(X)\subset
\LC^{\natural{n}}(X)\subset\LACG^{\natural{a}}(X).$
The first inclusion follows from Proposition \ref{altFUSS} below.
To prove the second inclusion, let $g\in \LC^{\natural{n}}(X)$ and
$[\xb,\xb+u]\subset \dom g$. The function $\varphi:t\in [0,1]\mapsto g(\xb+tu)$
is continuous, so by Fact \ref{nat-subdiv}\,(2)
and the definition of $\LC^{\natural{n}}(X)$,
$$D^+\varphi(t)=g^r_+(\xb_t;u)\le \gu(\xb_t;u)= g^r(\xb_t;u)=D_+\varphi(t) \text{ for nearly all } t\in [0,1].$$
Therefore, $D_+\varphi(t)$ is finite and $D_+\varphi(t)=D^+\varphi(t)$
nearly everywhere on $[0,1]$, which means that $\varphi$ is
right-differentiable nearly everywhere on $[0,1]$.
We conclude that $\varphi$ is ACG$_*$ on $[0,1]$ by Fact \ref{diff-ACG}\,(2).
Hence $g\in \LACG^{\natural{a}}(X)$.
A discussion of these classes of functions,
with examples, comments and variants,
is given in the last section.
\smallbreak
We say that a function $g:X\to\Rex$ is
\textit{radially Lipschitz continuous at a point $x\in \dom g$
in the direction $u\in X$}
if the restriction of $g$ to any open line segment
${]}\xb,\xb+u{[}$ containing $x$ is locally Lipschitz
at $x$, namely,
there exist $t_0>0$ and $\lambda>0$ such that
$$
y, z \in {]}x-t_0u,x+t_0u{[} \Longrightarrow g(z)-g(y)\leq
\lambda \|z-y\|.
$$

\begin{proposition}[Radial Lipschitz continuity of functions in
$\Lsc^{\natural\natural}(X)$]\label{altFUSS}
The restriction of every $g\in \Lsc^{\natural\natural}(X)$
to any line segment $[\xb,\xb+u]\subset \dom g$ is continuous at
the endpoints and locally Lipschitz at every $x\in {]}\xb,\xb+u{[}$;
in particular, $g^r(x;u)$ is finite for all
$x\in {]}\xb,\xb+u{[}$.
\end{proposition}

\begin{proof}
Let $x\in {]}\xb,\xb+u{[}$. We show that $g$ is locally Lipschitz
at $x$ relative to ${]}\xb,\xb+u{[}$.
Note that ${]}\xb,\xb+u{[}={]}\yb,\yb-u{[}$ for $\yb=\xb+u$.
By definition of the space $\Lsc^{\natural\natural}(X)$ we have
$\gsu(x;v)= g^r(x;v)<+\infty$ for $v=\pm u$. Hence,
there exists $\lambda\in\R$ such that, for $v=\pm u$,
$
\gsu(x;v)\le\lambda.
$
Let $x_t=x+tv$. Since $v+\alpha(x-x_t)= (1-\alpha t)v$, it follows that
for any $\alpha\ge 0$,
\begin{align*}
\limsup_{t\to 0} g^r(x_t;v)=\limsup_{t\to 0}\,(1-\alpha t)g^r(x_t;v)
&=\limsup_{t\to 0}\,g^r(x_t;v+\alpha(x-x_t))\\
&\le  \limsup_{x'\to x}\,g^r(x';v+\alpha(x-x')).
\end{align*}
Hence, for $v=\pm u$,
$$
\limsup_{t\to 0} g^r(x_t;v)\le \gsu(x;v)\le \lambda.
$$
We derive that there exists $t_0>0$ such that for $v=\pm u$,
\begin{equation}\label{five}
g^r(x';v)\le \lambda \text{ for all } x'\in {]}x-t_0u,x+t_0u{[}.
\end{equation}

Let $y,z$ be arbitrary points in ${]}x-t_0u,x+t_0u{[}$. There
exist $t_1\in [0,t_0{[}$ and $v\in\{u,-u\}$ such that $z=y+t_1v$.
Consider the lsc function $t\mapsto \varphi(t)=g(y+tv)-\lambda t$.
Using \eqref{five}, we see that
$$
D_+\varphi(t)=g^r(y+tv;v)-\lambda\le 0 \text{ for all $t\in [0,t_1]$}.
$$
The Monotonicity Theorem \ref{monotone-test}\,(1) then shows
that $\varphi(t_1)\le \varphi(0)$, in other words,
$$
g(z)-g(y)\le \lambda t_1=(\lambda/\|u\|)\|z-y\|.
$$
Thus the local Lipschitz property holds as long as
$y$ and $z$ belong to ${]}x-t_0u,x+t_0u{[}$.
From this it follows that
$-\lambda\le g^r(x;u)\le \lambda$, hence $g^r(x;u)$ is finite.
\smallbreak
We now prove that $g$ is continuous at $\xb$ relative to $[\xb,\xb+u]$.
The argument is similar. Since $\gsu(\xb;u)= g^r(\xb;u)<+\infty$,
there exists $\lambda\in\R$ such that $\gsu(\xb;u)\le\lambda.$
As above, we infer that
$$
\limsup_{t\searrow 0} g^r(\xb+tu;u)\le \gsu(\xb;u)\le \lambda.
$$
So there exists $t_0>0$ such that
\begin{equation*}
g^r(\xb+tu;u)\le \lambda \text{ for all } t\in [0,t_0{[}.
\end{equation*}
Then, we derive from the Monotonicity Theorem \ref{monotone-test}\,(1)
that, for every $t\in [0,t_0{[}$,
$$
g(\xb+tu)-g(\xb)\le \lambda t,
$$
proving that $g$ is continuous at $\xb$ relative to $[\xb,\xb+u]$.
Since $[\xb,\xb+u]=[\yb,\yb-u]$ for $\yb=\xb+u$, the continuity
of $g$ at $\xb+u$  relative to $[\xb,\xb+u]$ follows as well.
\end{proof}

Typical examples of functions in $\Lsc^{\natural\natural}(X)$
are the proper lsc convex functions (see the last section).
As an illustration of Proposition \ref{altFUSS}, consider
the proper lsc convex function $g$ defined on $X=\R$ by
$g(x)=-\sqrt{x}$ for $x\in [0,1]$ and $g(x)=+\infty$ otherwise.
Then $g$ is continuous on $[0,1]$, locally Lipschitz at every
point in ${]}0,1{[}$ but not locally Lipschitz at $\xb=0$.
\smallbreak
Let $\Fcal(X)$ be a class of lsc functions on $X$.
We say that a subclass $\mathcal{G}(X)\subset \Fcal(X)$ is \textit{subdifferentially
determined in $\Fcal(X)$} if for every $g\in \mathcal{G}(X)$, $f\in \Fcal(X)$
and $\Omega\subset X$ open convex with $\Omega\cap\dom f\ne\emptyset$,
one has
\begin{equation*}\label{subdiffdet-revised}
\del f(x)\subset \del g(x) \text{ for all } x\in \Omega
\Longrightarrow f=g + Const  \text{ on  }\Omega\cap\dom f.
\end{equation*}

\smallbreak

Each version of the Monotonicity Theorem \ref{monotone-test}
naturally leads to a corresponding
version for the subdifferential determination property:

\begin{theorem} [Subdifferential determination of functions]
\label{determination-revisited}
Let $X$ be a Banach space.
\smallbreak
{\rm (1)} The class $\Lsc^{\natural\natural}(X)$ is subdifferentially
determined in $\Lsc(X)$.
\smallbreak
{\rm (2)} The class $\NFUSS(X)$ is subdifferentially determined in $\LC(X)$.
\smallbreak
{\rm (3)} The class $\AFUSS(X)$ is subdifferentially determined in $\LACG(X)$.
\end{theorem}

\begin{proof}
The structure of the proof is the same for each case and is similar to
that of \cite[Theorem 10]{Las18b}.
We give the details for the case (1) and a sketch for
the other (simpler) cases.
\medbreak
\textit{Case} (1).
Let $g\in \Lsc^{\natural\natural}(X)$ and $f\in \Lsc(X)$,
and let $\Omega\subset X$ be
an open convex subset with $\Omega\cap\dom f\ne\emptyset$. Assume
\begin{equation}\label{final0}
\del f(x)\subset \del g(x) \quad \mbox{for all } x\in\Omega.
\end{equation}
Without loss of generality, we may consider that
$\Omega\cap \dom f$ contains two distinct points.
By Fact \ref{densitytheorem}, the set
$\Omega\cap \dom \del f$ is dense in $\Omega\cap \dom f$, it
therefore also contains two distinct points.
Observe that by \eqref{final0} $\Omega\cap \dom \del f
=\Omega\cap \dom \del f\cap \dom \del g\subset
\Omega\cap \dom f\cap\dom g$.
\smallbreak

\textit{First step.}
Let $\xb\in \Omega\cap \dom \del f$ and
$\yb\in \Omega\cap \dom f\cap\dom g$ with $\xb\ne\yb$.
Note that $[\xb,\yb]\subset \Omega\cap \dom g$.
Put $u:=\yb-\xb$ and let $t\in [0,1{[}$. Then
$\xb_t:=\xb+tu\in [\xb,\xb+u{[}=[\xb,\yb{[}\subset \Omega\cap \dom g$.
If $\xb_t\not\in\dom f$, $f^r(\xb_t;u)=-\infty$, hence
$f^r(\xb_t;u)\le g^r(\xb_t;u)$. Otherwise,
$\xb_t\in\Omega\cap \dom f\cap\dom g$, so by
Fact \ref{nat-subdiff}
\begin{align*}
\fsu(\xb_t;u)=
\inf_{\alpha\ge 0}\limsup_{x\to \xb_t}\,f^\del(x;u+\alpha(\xb_t-x))\\
\gsu(\xb_t;u)=\inf_{\alpha\ge 0}\limsup_{x\to \xb_t}\,g^\del (x;u+\alpha(\xb_t-x)),
\end{align*}
which entails from \eqref{final0} that 
$\fsu(\xb_t;u)\le \gsu(\xb_t;u).$
But $f^r(\xb_t;u)\le \fsu(\xb_t;u)$ by definition of $\fsu(\xb_t;u)$
and $\gsu(\xb_t;u)= g^r(\xb_t;u)$ by assumption on $g$.
Hence $f^r(\xb_t;u)\le g^r(\xb_t;u)$ also in this case.
Thus, we have just shown that
\begin{equation}\label{fin4}
f^r(\xb_t;u)\le g^r(\xb_t;u) \text{ for all $t\in [0,1{[}$.}
\end{equation}
\ps
\textit{Second step.} 
By Proposition \ref{altFUSS}, $g^r(\xb_t;u)$ is finite for
every $t\in {]}0,1{[}$. On the other hand,
$g^r(\xb;u)< +\infty$ and since $\xb\in \Omega\cap \dom \del g$,
we infer that $g^r(\xb;u)=\gsu(\xb;u)\ge g^\del (\xb;u)>-\infty$,
so $g^r(\xb;u)$ is finite as well.
Rewriting \eqref{fin4} with the functions $\varphi:t\in [0,1]\mapsto f(\xb_t)$
and $\gamma:t\in [0,1]\mapsto g(\xb_t)$,
we get
\begin{equation}\label{fin4b}
D_+\varphi(t)\le D_+\gamma(t) \text{ for all $t\in [0,1{[}$}.
\end{equation}
Since $D_+\gamma(t)$ is finite everywhere on $[0,1{[}$, one has
$D_+(\varphi-\gamma)(t)\le D_+\varphi(t)- D_+\gamma(t)$ everywhere on $[0,1{[}$,
hence \eqref{fin4b} entails
\begin{equation*}
D_+(\varphi-\gamma)(t)\le 0  \text{ for all $t\in [0,1{[}$}.
\end{equation*}
Note that $\varphi$ is lsc and $\gamma$ is continuous by Proposition \ref{altFUSS}, so the function $\varphi-\gamma$ is lsc.
Applying the Monotonicity Theorem \ref{monotone-test}\,(1),
we obtain that
$(\varphi-\gamma)(1)\le (\varphi-\gamma)(0)$.
Finally we have proved that
\begin{equation}\label{fin1}
f(\yb)-f(\xb)\le g(\yb)-g(\xb)
\text{ for all $\xb\in \Omega\cap \dom \del f$
and $\yb\in \Omega\cap \dom f\cap\dom g$.}
\end{equation}
This inequality can be extended to all $\xb\in \Omega\cap \dom f$
and $\yb\in \Omega\cap \dom f\cap\dom g$. Indeed,
by Fact \ref{densitytheorem}, there is
a sequence $(\xb_n)_n\subset \Omega\cap \dom \del f$
such that $\xb_n\to \xb$
and $f(\xb_n)\to f(\xb)$. Since  $g$ is lower semicontinuous at $\xb$,
passing to the limit in the inequality
$$f(\yb)-f(\xb_n)\le g(\yb)-g(\xb_n),$$
we see that \eqref{fin1} holds for all $\xb\in \Omega\cap \dom f$. 
From this we derive that every point $\xb$ in $\Omega\cap\dom f$
belongs to $\Omega\cap\dom g$, so $\Omega\cap \dom f\cap\dom g=\Omega\cap \dom f$.
We conclude that
$$
f(\yb)-f(\xb)\le g(\yb)-g(\xb) \text{ for all } \xb,\yb \in \Omega\cap\dom f,
$$
hence in fact
$$
f(\yb)-f(\xb)= g(\yb)-g(\xb) \text{ for all } \xb,\yb \in \Omega\cap\dom f,
$$
which means that $f=g + {\rm const}$  on $\Omega\cap \dom f$.

\medbreak
\textit{Case} (2) (\textit{Case} (3), respectively).
Let $g\in \NFUSS(X)$ and $f\in \LC(X)$
($g\in \AFUSS(X)$ and $f\in \LACG(X)$, respectively),
and let $\Omega\subset X$ be an open convex subset with
$\Omega\cap\dom f\ne\emptyset$. Assume that the inclusion \eqref{final0} holds.

Let $\xb\in \Omega\cap \dom \del f=
\Omega\cap \dom \del f\cap \dom \del g$
and $\yb\in \Omega\cap \dom f\cap\dom g$ with $\xb\ne\yb$.
Note that $[\xb,\yb]\subset \Omega\cap \dom g$.
As in Case (1), we let $u:=\yb-\xb$ and for $t\in [0,1]$,
$\xb_t:=\xb+tu\in \Omega\cap \dom g$.
If $\xb_t\not\in\dom f$, $f^r_+(\xb_t;u)=-\infty$, hence
$f^r_+(\xb_t;u)\le g^r(\xb_t;u)$. Otherwise,
$\xb_t\in\Omega\cap \dom f\cap\dom g$, so proceeding as
in Case (1), we derive from
Fact \ref{nat-subdiff} and \eqref{final0} that
$\fu(\xb_t;u)\le \gu(\xb_t;u).$
But by Fact \ref{nat-subdiv}\,(2), $f^r_+(\xb_t;u)\le \fu(\xb_t;u)$,
and by assumption on $g$, $\gu(\xb_t;u)= g^r(\xb_t;u)\in\R$ nearly (almost, respectively) everywhere on $[0,1]$.
Hence finally
\begin{equation}\label{fin44}
g^r(\xb_t;u) \text{ is finite and }
f^r_+(\xb_t;u)\le g^r(\xb_t;u)
\text{ for nearly (almost, resp.) all $t\in [0,1]$.}
\end{equation}
Rewriting \eqref{fin44} with the functions
$\varphi:t\in [0,1]\mapsto f(\xb_t)$ and
$\gamma:t\in [0,1]\mapsto g(\xb_t)$,
we get
\begin{equation}\label{fin44b}
D_+\gamma(t) \text{ is finite and }
D^+\varphi(t)\le D_+\gamma(t)
\text{ for nearly (almost, resp.) all $t\in [0,1]$,}
\end{equation}
hence,
\begin{equation}\label{fin44c}
D^+(\varphi-\gamma)(t)\le 0
\text{ for nearly (almost, resp.) all $t\in [0,1]$.}
\end{equation}
Applying the appropriate version of the Monotonicity
Theorem \ref{monotone-test} we conclude that
$\varphi(1)-\gamma(1)\le \varphi(0)-\gamma(0)$.
Thus we have proved that
$$
f(\yb)-f(\xb)\le g(\yb)-g(\xb)
\text{ for all $\xb\in \Omega\cap \dom \del f$
and $\yb\in \Omega\cap \dom f\cap\dom g$.}
$$
The rest of the proof goes in the same way as in Case (1).
\end{proof}

We now proceed with the subdifferential representation issue.
To this end, we introduce the class of almost upper semismooth functions
in the restricted sense:

\vspace*{-0.8cm}
\begin{multline*}
\LACG_*^{\natural{a}}(X)=
\{ g\in \LACG_*(X)\tq \dom g \text{ is convex and }
\forall [\xb,\xb+u]\subset \dom g,\ \xb_t=\xb+tu,\\ 
\gu(\xb_t;u)= g^r(\xb_t;u) \text{ for almost all } t\in [0,1]\}.
\end{multline*}

One has $\LC^{\natural{n}}(X)\subset\LACG_*^{\natural{a}}(X)
\subset\LACG^{\natural{a}}(X)$.
Indeed, we have already observed that  
each function in $\LC^{\natural{n}}(X)$ belongs to $\LACG_*(X)$
so the first inclusion holds.
The second inclusion follows from the fact that
for any $g\in \LACG_*^{\natural{a}}(X)$ and $[\xb,\xb+u]\subset \dom g$,
the function $\varphi:t\in [0,1]\mapsto g(\xb+tu)$
is differentiable almost everywhere on $[0,1]$ by Fact \ref{diff-ACG}\,(1),
so $g^r(\xb_t;u)$ is finite for almost all $t\in [0,1]$.

\begin{theorem}[Subdifferential representation of functions] \label{recover}
Let $X$ be a Banach space.
Any $g\in \LACG_*^{\natural{a}}(X)$ can be
represented through its subdifferential via the integration formula
\begin{gather*}
g(\xb+u)-g(\xb)=\int_{0}^1 \gu(x_t;u) dt,
\quad \forall [\xb, \xb+u{[}\subset \dom g,\ \xb_t=\xb+t u,
\end{gather*}
where
$\displaystyle \gu(\xb_t;u)=\inf_{\alpha\ge 0}\limsup_{x\to_u\xb_t}
g^\del(x;u+\alpha (\xb_t-x))$ for every $t\in [0,1]$.
\end{theorem}

\begin{proof}
The function $\varphi:t\mapsto g(\xb_t)$
is ACG$_*$ on $[0,1]$ and $D_+\varphi(t)=g^r(\xb_t;u)$.
It therefore follows from the Subderivative Representation
Theorem \ref{recov-subdiv-ACG} that
$$g(\xb+u)-g(\xb)=\int_{0}^1 g^r(\xb_t;u) dt.$$
But $g^r(\xb_t;u)=\gu(\xb_t;u)$
almost everywhere on $[0,1]$ and by Fact \ref{nat-subdiff},
\begin{align*}
\gu(\xb_t;u)=\inf_{\alpha\ge 0}\limsup_{x\to_u\xb_t}
g^\del(x;u+\alpha (\xb_t-x)).\tag*{\qedhere}
\end{align*}
\end{proof}

\section{Examples, comments and variants}\label{examples}

\subsection{The space \texorpdfstring{$\Lsc^{\natural\natural}(X)$}{sss}}\label{fsuss}

Let $g:X\to\R$ be locally Lipschitz on a open convex subset
$U\subset X$. Then, for every $x\in U$ and $u\in X$,
$g^\circ(x;u)$ is finite and $g^\circ(x;u)=\gsu(x;u)$
(see the diagram \eqref{diagramlip}). So the
equality $g^r(x;u)=\gsu(x;u)$ for every $u\in X$ is equivalent to the
(Clarke) regularity of $g$ at $x$, i.e.\ $g^r(x;u)=g^\circ(x;u)$
for every $u\in X$.
In other words, the locally Lipschitz functions in $\Lsc^{\natural\natural}(X)$
are precisely the (Clarke) regular functions.
\smallbreak
Besides the locally Lipschitz regular functions,
the space $\Lsc^{\natural\natural}(X)$ contains
the proper lsc convex functions,
the proper lsc approximately convex functions
(hence also the lower-C$^1$ functions)
and (more generally) the directionally stable functions in the sense
of Thibault-Zagrodny \cite{TZ95}.
See \cite{Las18b} for proofs and discussion.
\smallbreak
We don't know whether the space $\Lsc^{\natural\natural}(X)$
contains the lsc radially Lipschitz continuous functions
which are regular in the sense of Rockafellar,
i.e.\ $g^d(x;u)=g^\uparrow(x;u)$ for every $u\in X$.
We recall that for a convex lsc $g$ one has,
for every $x, x+u \in \dom g$,
$$
g^d(x;u)=g^\uparrow(x;u)\le g^r(x;u)=\gsu(x;u)<+\infty
$$
where the inequality $\le$ may be strict.
\subsection{The space \texorpdfstring{$\NFUSS(X)$}{nss}}\label{nfuss}

The space $\NFUSS(X)$ contains the Mifflin semismooth functions
like $x\in\R\mapsto -|x|$ (see \cite{Las18b}).
It also contains non-locally Lipschitz functions like
$x\in\R\mapsto -\sqrt{|x|}$
or $x\in\R\mapsto \sqrt{|x|}$,
and even non-absolutely continuous functions like
$f:\R\to\R$ given by
$$
f(x):=\left\{
\begin{array}{ll}
x\sin (1/x) & \mbox{if~~} x\ne 0\\
0 & \mbox{if~~} x=0
\end{array}
\right..
$$
These functions are not in $\Lsc^{\natural\natural}(X)$.
\subsection{The space \texorpdfstring{$\AUSS(X)$}{ass}}\label{auss}

The following classes of locally Lipschitz
functions are considered in Thibault-Zagrodny \cite{TZ10}:

- the \textit{segmentwise essentially smooth} functions
\cite[p.\ 2305]{TZ10},
that is,
locally Lipschitz functions $g$ defined on a nonempty open convex subset 
$\Omega\subset X$, such that for every $\xb,u\in X$ with
$[\xb,\xb+u]\subset\Omega$,
\begin{equation}\label{defes}
g^\circ(\xb+tu;u)= -g^\circ(\xb+tu;-u)\mbox{ for almost all $t\in [0,1]$,}
\end{equation}

- the \textit{segmentwise essentially subregular} functions
\cite[Definition 4.6]{TZ10},
that is, locally Lipschitz functions $g$ such that, instead of \eqref{defes}
one has
\begin{equation}\label{dereg}
g^\circ(\xb+tu;u)= g^r(\xb+tu;u)\mbox{ for almost all $t\in [0,1]$.}
\end{equation}

In fact, \eqref{defes} and \eqref{dereg} are equivalent
(see the proof of \cite[Lemma 2.1]{BM98a}),
so the two classes are identical.
They contain the class of \textit{arcwise essentially smooth}
functions previously studied by Borwein-Moors \cite{BM98a}.
A remarkable feature of these classes is that they are stable by composition,
addition and multiplication. For more details,
see \cite{BM98a,TZ10,Zaj12} and the references therein.

These functions are contained in a more sophisticated
class of functions introduced by L. Thibault and D. Zagrodny in \cite{TZ10}:
given a subdifferential $\partial$, a lsc function
$g:X\to\xR$ is called 
\textit{$\del$-essentially directionally smooth} (\textit{eds} for
short) on a nonempty open convex subset $\Omega\subset X$,
provided that (simplified version)
for every $u, v\in \Omega \cap\dom g$
with $v\ne u$, the following properties hold:
\begin{itemize}\itemsep-1ex\topsep0pt
\item[(i)] the function $t\mapsto g_{u,v} (t):= g(u+t(v-u))$ is finite and
continuous on $[0, 1]$;
\item[(ii)] there are real numbers $0=t_0<\cdots<t_p=1$ such that the
function $t\mapsto g_{u,v} (t)$ is absolutely continuous on each closed
interval included in $[0,1]\setminus \{t_0,t_1,\ldots,t_p\}$;
\item[(iii)] for every $\mu>0$ there exists a subset $T_\mu \subset [0, 1]$
of full Lebesgue measure (i.e.\ of Lebesgue measure 1)
such that for every $t\in T_\mu$ and every
sequence $((x_k, x^*_k ))_k \subset \partial g$ with $x_k \to x(t):=u+t(v-u))$,
there is some $w\in {]}x(t), v]$ for which
\begin{equation*}
\limsup_{k\to \infty} \la x^*_k, w-x_k\ra \le \|w-x(t)\| \left (\|v-u\|^{-1} g^r_{u;v}(t;1)+\mu\right ).
\end{equation*}
\end{itemize}

\begin{proposition}[eds versus $\AUSS(X)$]\label{eds-uss}
Each eds function $g:X\to\xR$ belongs to \LACG$_*$(X) and
satisfies: for every $[\xb,\xb+u]
\subset \dom g$,
$\gsu(\xb+tu;u)= g^r(\xb+tu;u)$ for almost all $t\in [0,1]$. Hence,
the class of eds functions is contained in $\AUSS(X)$.
\end{proposition}

\begin{proof}
Let $g:X\to\Rex$ be $\del$-essentially directionally smooth on $\Omega$.
Let $\xb\in\Omega\cap\dom g$ and $u\in X$ so that
$\xb+u\in\Omega\cap\dom g$. We apply the above definition with the
pair $(\xb,\xb+u)$ in lieu of $(u,v)$.
For $t\in [0,1]$, we set $\xb_t:=\xb+tu$ and $g_{\xb,u}(t):=g(\xb_t)$.
The conditions (i) and (ii) imply that $t\mapsto g_{\xb,u}(t)$ is
ACG$_*$ on $[0,1]$ (see the observation before Fact \ref{diff-ACG}), hence $g$ belongs to \LACG$_*$(X).
Let $\mu>0$. By (iii),
there exists a subset $T_\mu \subset [0, 1{[}$ of full Lebesgue measure such that
for every $t\in T_\mu$ and every sequence $((x_n,x^*_n))_n \subset \del g$
with $x_n \to \xb_t$ there is some $w\in {]}\xb_t, \xb+u]$ for which
\begin{equation}\label{eds1}
\limsup_{n\to \infty} \la x^*_n, w-x_n\ra \le \|w-\xb_t\| \left (\|u\|^{-1}
g^r_{\xb,u}(t;1)+\mu\right ).
\end{equation}
Since $w\in {]}\xb_t, \xb+u]$ there exists $\tau>0$ such that $w=\xb_t+\tau u$.
Then \eqref{eds1} can be rewritten as
$$
\limsup_{n\to\infty}\, \la x^*_n,\tau u+\xb_t-x_n\ra \le \tau\|u\|
\left (\|u\|^{-1}g^r(\xb_t;u)+\mu\right )=\tau (g^r(\xb_t;u)+\mu\|u\|).
$$
Hence, for any $t \in T_\mu$ and  every $((x_n,x^*_n))_n \subset \del g$
with $x_n \to \xb_t$ there is $\tau>0$ such that
\begin{equation*}
\limsup_{n\to\infty}\, \la x^*_n, u+\frac{\xb_t-x_n}{\tau}\ra \le
g^r(\xb_t;u)+\mu\|u\|.
\end{equation*}
Setting $\alpha:=1/\tau>0$, we derive that 
for every $\mu>0$ there exists a subset $T_\mu \subset [0, 1{[}$
of full Lebesgue measure such that for every $t\in T_\mu$,
\begin{equation*}
\inf_{\alpha> 0}\limsup_{x\to \xb_t} g^\del (x; u +\alpha (\xb_t-x))\le g^r(\xb_t;u)+\mu\|u\|.
\end{equation*}
Consider the subset $T \subset [0, 1{[}$
of full Lebesgue measure defined by $T:=\bigcap_{n\in \N^*} T_{1/n}$.
Then, for every $t\in T$ it holds
\begin{equation}\label{eds2}
\inf_{\alpha> 0}\limsup_{x\to \xb_t} g^\del (x; u +\alpha (\xb_t-x))\le
g^r(\xb_t;u).
\end{equation}
By Fact \ref{nat-subdiff}, the left-hand side of \eqref{eds2} is
equal to $\gsu(\xb_t;u)$. Hence, $\gsu(\xb_t;u)= g^r(\xb_t;u)$
for almost all $t\in [0,1]$. A fortiori, $\gu(\xb_t;u)= g^r(\xb_t;u)$
for almost all $t\in [0,1]$. The proof is complete. 
\end{proof}

\subsection{Continuous variant}\label{variant}

When the functions are \textit{continuous}, a more refined subdifferential
determination property can be established with a simpler proof.
Let $G\subset X$ be a nonempty open convex subset.
We denote by $\CLACG(G)$ the class of all real-valued Continuous functions on $G$
whose restrictions to Line segments $[a,b]\subset G$ are ACG, and
we consider its subclass of densely almost upper semismooth functions
defined by:
\vspace*{-0.2cm}
\begin{multline*}
\CLACG^{\natural{ad}}(G)=
\{ g\in \CLACG(G)\tq 
\forall [\xb,\xb+u]\subset G\ \exists \xb_n\to \xb,\, u_n\to u : 
\forall n, \\
g^r(\xb_n+tu_n;u_n)\in\R \text{ and }
\gu(\xb_n+tu_n;u_n)= g^r(\xb_n+tu_n;u_n) \text{ for almost all } t\in [0,1]\}.
\end{multline*}

\begin{theorem}[Subdifferential Determination -- Continuous variant]
\label{subdiffdetcont}
Let $X$ be a Banach space. The class
$\CLACG^{\natural{ad}}(G)$ is subdifferentially determined in $\CLACG(G)$.
\end{theorem}

\begin{proof}
Let $g\in \CLACG^{\natural{ad}}(G)$ and $f\in \CLACG(G)$,
and let $\Omega\subset G$ be a nonempty open convex subset.
Assume
\begin{equation*}
\del f(x)\subset \del g(x) \quad \mbox{for all } x\in\Omega.
\end{equation*}
Let $\xb,\xb+u\in \Omega$.
Let $\xb_n\to \xb$ and $u_n\to u$ such that
$\gu(\xb_n+tu_n;u_n)= g^r(\xb_n+tu_n;u_n)\in\R$ for almost all $t\in [0,1]$.
Proceeding as in Theorem \ref{determination-revisited}\,(3),
we derive that, for every $n\in \N$,
\begin{equation*}
f^r_+(\xb_n+tu_n;u_n)\le g^r(\xb_n+tu_n;u_n)\in\R
\text{ for almost all $t\in [0,1]$,}
\end{equation*}
which leads to
$$
f(\xb_n+u_n)-f(\xb_n)\le g(\xb_n+u_n)-g(\xb_n)
\text{ for every }n\in \N.
$$
Since $f$ and $g$ are continuous,
passing to the limit, we get $f(\xb+u)-f(\xb)\le g(\xb+u)-g(\xb)$.
It then follows that $f-g$ is constant on $\Omega$. The proof is complete.
\end{proof}

Now, let $\mathcal{S}_e(G)$ denote the class of
\textit{essentially smooth} functions studied by Borwein-Moors
\cite{BM98a,BM98b},
that is, the locally Lipschitz functions $g$ on $G$ such that 
for each $u\in X$,
\begin{equation}\label{defessmooth}
B_u:=\{ x\in G\tq g^\circ(x;u)\ne -g^\circ(x;-u)\}\ \text{is a
\textit{Haar-null subset} of } X.
\end{equation}

\begin{proposition}[$\mathcal{S}_e(G)$ versus $\CLACG^{\natural{ad}}(G)$]
\label{BMclass}
Let $G\subset X$ be a nonempty open convex subset of a Banach space $X$.
Each $g \in \mathcal{S}_e(G)$ satisfies: for every $u\in X$
there is a dense subset $D_u$ of $G$ such that for every $w\in D_u$
with $[w,w+u]\subset G$,
$g^{\natural\natural}(w+tu;u)= g^r(w+tu;u)$ for almost all $t\in [0,1]$.
Hence,
$\mathcal{S}_e(G)\subset \CLACG^{\natural{ad}}(G)$.
\end{proposition}
\begin{proof}
Let $g \in \mathcal{S}_e(G)$, $0\ne u \in X$, and let $W_u$
be a topological complement of span$\{u\}$.
Applying Fact \ref{null-th} given below to the Borel Haar null set $B_u$
defined in \eqref{defessmooth}, we obtain that there is a dense set
$S_u$ in $W_u$ such that for every $w\in S_u$,
\begin{equation}\label{defesgen}
g^\circ(w+tu;u)= -g^\circ(w+tu;-u)\mbox{ for almost all $t\in \R$.}
\end{equation}
Then the set $D_u:=S_u+\text{span$\{u\}$}$ is a dense subset of $G$
such that for every $w\in D_u$ the relation \eqref{defesgen} holds, or equivalently (see the discussion at the beginning of
Subsection \ref{auss}),
\begin{equation*}
g^\circ(w+tu;u)= g^r(w+tu;u)\mbox{ for almost all $t\in \R$,}
\end{equation*}
which implies (see the diagram \eqref{diagram})
\begin{equation*}
g^{\natural\natural}(w+tu;u)=g^{\natural}(w+tu;u)
= g^r(w+tu;u)\mbox{ for almost all $t\in [0,1]$.}
\end{equation*}
In particular, $g\in \CLACG^{\natural{ad}}(G)$.
\end{proof}

\begin{fact}[{\cite[Lemma 2.4]{Zaj12} and \cite[Theorem 2.4]{BM98b}}]
\label{null-th}
Let $X$ be a Banach space, $0\ne u \in X$, and let $W$ be a topological
complement of {\rm span}$\{u\}$. Let $B\subset X$ be a Borel Haar-null set.
Then there exists a set $S \subset W$ dense in $W$ such that the set
$\{t \in\R : w + tu \in B\}$ is Lebesgue-null for each $w\in S$.
\end{fact}
{\small
}
\end{document}